\documentclass[a4paper, 10pt]{amsart}

\usepackage{amsmath, amsfonts, amssymb}
\usepackage{amsthm}
\usepackage{esint}
\usepackage[
bookmarks=false]{hyperref}
\usepackage{cite}
\usepackage{array}

\usepackage{amsaddr}
\usepackage{enumerate}
\usepackage{pgfplots}

\title{\large{On noncompactness of the $\dbar$-Neumann problem\\ on pseudoconvex domains in $\C^3$}}

\author{Gian Maria Dall'Ara}
\address{Universit\"at Wien, Vienna, Austria}
\email{gianmaria.dallara@univie.ac.at}
\thanks{Research supported by FWF-project P28154.}

\date{\today}

\newcommand{\R}{\mathbb{R}}
\newcommand{\N}{\mathbb{N}}

\newcommand{\C}{\mathbb{C}}

\newcommand{\OO}{\mathcal{O}}

\newcommand{\dbar}{\overline{\partial}}

\newcommand{\be}{\begin{equation*}}
\newcommand{\ee}{\end{equation*}}
\newcommand{\bel}{\begin{equation}}
\newcommand{\eel}{\end{equation}}
\newcommand{\bee}{\begin{eqnarray*}}
\newcommand{\eee}{\end{eqnarray*}}
\newcommand{\eps}{\varepsilon}

\newtheorem{thm}{Theorem}
\newtheorem{lem}[thm]{Lemma}
\newtheorem{prop}[thm]{Proposition}

\begin{document}

\maketitle

\begin{abstract}
We prove that a smooth bounded pseudoconvex domain $\Omega\subseteq\C^3$ with a one-dimensional complex manifold $M$ in its boundary has a noncompact Neumann operator on $(0,1)$-forms, under the additional assumption that $b\Omega$ has finite regular D'Angelo $2$-type at a point of $M$.
\end{abstract}

\section{Introduction and main result}

In this paper we deal with the following question, raised for instance in \cite{fu-straube-comp}: is it true that any bounded smooth pseudoconvex domain in $\C^n$ whose boundary contains a $q$-dimensional complex manifold $M$ necessarily has a noncompact $\dbar$-Neumann operator $N_q$ ($1\leq q\leq n-1$)? We refer the reader to \cite{straube-book} and \cite{fu-straube-comp} for the relevant background on the $\dbar$-Neumann problem in general, and the importance of compactness in particular.

Before stating our own contribution, let us list the main results presently known.

\begin{itemize}

\item[(a)] Krantz discussed in \cite{krantz-comp} the case of the bidisc in $\C^2$ (which strictly speaking is not smooth, but one may round the edges to fall under our assumptions). This is somehow the archetipal example for failure of compactness, since all the arguments known to establish it show first that the domain in consideration has a local ``almost-product'', or more precisely ``fibered'', structure (see Section \ref{structure-sec} below, in particular the discussion after Lemma \ref{bedford-fornaess-lemma}).

\item[(b)] The answer is affirmative in $\C^2$, as proved for general Lipschitz domains in \cite{fu-straube-comp}. The authors of that paper claim that this result has been part of the folklore for many years, with unpublished proofs by Catlin and Christ.

\item[(c)] Matheos exhibited a $2$-dimensional bounded smooth pseudoconvex Hartogs domain with no $1$-dimensional complex manifold in the boundary, which nonetheless has noncompact $\dbar$-Neumann operator (see \cite{matheos} or Theorem $4.25$ of \cite{straube-book}), disproving the reverse implication to the one we are considering. Christ and Fu (see \cite{christ-fu}) later fully characterized compactness of $N_1$ on $2$-dimensional bounded smooth pseudoconvex Hartogs domains.

\item[(d)] Fu and Straube showed in \cite{fu-straube-comp} that on bounded locally convexifiable domains (in any dimension) $N_q$ is compact if and only if there are no $q$-dimensional complex manifolds in the boundary (see also section $4.9$ of \cite{straube-book}). 

\item[(e)] \c{S}ahuto\u{g}lu and Straube proved failure of compactness of $N_q$ for smooth pseudoconvex domains in $\C^n$ with a $q$-dimensional manifold $M$ in the boundary, under the additional assumption that at a point of $M$ the boundary is strictly pseudoconvex in the directions transverse to $M$ (see \cite{sahutoglu-straube} for the case $q=1$, and \cite{sahutoglu-thesis} or Theorem $4.21$ of \cite{straube-book} for the general case). Notice that if $q=n-1$ this hypothesis is empty and the problem is completely solved.

\end{itemize}

 The transversality condition in the last result is, in the words of one of the authors, ``somewhat curious'' as ``one would expect a flatter boundary to be even more conducive to noncompactness'' (p. 100 of \cite{straube-book}). Our goal in the present paper is to weaken significantly this assumption in the $3$-dimensional case (for $q=1$, since the case $q=2$ is settled), requiring only a finite-type condition in the directions transverse to the complex submanifold. More precisely, we prove the following:

\begin{thm}\label{main-thm}
Let $\Omega\subseteq\C^3$ be a bounded smooth pseudoconvex domain with a one-dimensional complex manifold $D$ in the boundary. Assume that there exists $P\in D$ such that the maximum order of contact of any $2$-dimensional complex manifold with $b\Omega$ at $P$ is bounded. Then the Neumann operator $N_1$ is not compact.
\end{thm}

A few comments may help to clarify the scope of Theorem \ref{main-thm}.

\begin{itemize}
\item[(a)] The boundedness of the order of contact of complex hypersurfaces with $b\Omega$ at $P$ means that there exists a $k_0<+\infty$ for which the following holds: if $U$ is a neighborhood of $0$ in $\C^2$, $h:U\rightarrow\C^3$ is a holomorphic mapping with Jacobian of rank $2$ at $0$ such that $h(0)=P$, and \be
\rho(h(w))=\OO(|w|^k)\qquad\forall w\in U,
\ee then $k\leq k_0$. Here $\rho$ is any defining function of $\Omega$. This condition is called \emph{finite regular $2$-type} by D'Angelo (see, e.g., \cite{dangelo-kohn}).
\item[(b)] Since $\Omega$ contains a one-dimensional complex manifold passing through $P$, our assumption is really a finite-type assumption in the directions transverse to $D$. See Lemma \ref{structure-lemma} below for a precise statement.
\item[(c)] The most dramatic way in which our assumption may be violated is if a $2$-dimensional complex manifold containing $D$ lies in the boundary, a condition equivalent to what is known as \emph{non-minimality} of the boundary (at the point $P$) in CR geometry (see Theorem $1.5.15$ of \cite{ber-book}). Notice, however, that if this is the case the result of \c{S}ahuto\u{g}lu and Straube  may be applied to conclude that $N_2$, the $\dbar$-Neumann operator on $(0,2)$-forms, fails to be compact, and this implies the non-compactness of $N_1$ too (Proposition $4.5$ of \cite{straube-book}). 

\item[(d)] Another way to look at this result is to observe that to give a positive answer to the problem we stated at the beginning of the introduction, one would like to replace our assumption with compactness of $N_2$ (again because of the percolation of compactness up the $\dbar$-complex). \newline By results of Kohn \cite{kohn} and Greiner \cite{greiner}, our hypothesis is instead equivalent to subellipticity at $P$ at the level of $(0,2)$-forms, which is slightly stronger.

\item[(e)] A hypothetical counterexample, i.e., a domain with a one-dimensional complex manifold $M$ in its boundary and compact Neumann operator on $(0,1)$-forms, must have the property that $M$ is made of minimal points satisfying an infinite-type condition transversally. Notice that the boundary of such a domain cannot be locally real-analytic around points of $M$. 
\end{itemize}

\subsection{Structure of the paper}
The rest of the paper is devoted to the proof of Theorem \ref{main-thm}. It follows the ``plan'' of \cite{sahutoglu-straube} and \cite{fu-straube-conv}, aiming to insert a smaller domain inside $\Omega$ that is tangent to $b\Omega$ along $D$ and has a product structure in appropriate coordinates. Morally, it is this product structure that causes the failure of compactness, but one must choose the smaller domain in such a way that its boundary has a sufficiently high order of contact with $b\Omega$, in order to transfer the non-compactness to the original domain $\Omega$. Under the assumptions of \cite{sahutoglu-straube}, the minimal order of contact $2$ is enough, and the construction boils down to a local structure lemma for domains with complex manifolds in the boundary which, in the $q=1$ case, follows from a result of Bedford and Forn\ae ss (Lemma $1$ of \cite{bedford-fornaess}). The key step and the main novelty in our proof is a refinement of this lemma in dimension $3$ (under our more general finite $2$-type assumption) that allows to increase this order of contact up to any finite value. Section \ref{structure-sec} is devoted to its proof, that is a sort of parametric version of a normalization procedure used in the study of $2$-dimensional finite-type pseudoconvex domains (see, e.g., \cite{catlin-invariant}). It relies on an anisotropic scaling argument of a kind often used in harmonic and complex analysis. For the sake of clarity, we devote Appendices \ref{aniso-app} and \ref{pscvx-app} to a rapid discussion of a few results we need in our proof.

The key technical tool to transfer non-compactness of $N_1$ from the smaller domain to the original one is a ``noncompactness'' lemma, of which we give in Section \ref{noncpt-sec} the most appropriate formulation for our argument. The assumptions of the lemma involve the behaviour on the diagonal of the Bergman kernel of a domain and a tangent subdomain one dimension less than the one for which one wants to establish non-compactness. The main reason for the hypothesis in our theorem is that Bergman kernels are very well understood in dimension $2$ under a finite-type assumption, thanks to work of Catlin, as we recall in Section \ref{bergman-sec}, while a general understanding lacks close to infinite type points and in higher dimension (despite significant progress: see, e.g., \cite{mcneal-stein-bergman}, \cite{charpentier-dupain-bergman}, and \cite{nagel-pramanik-diagonal}).

In Section \ref{proof-sec} we put all the ingredients together and complete the proof.

\section{Local structure of a 3-dim. pseudoconvex domain with a one-dimensional complex manifold in the boundary}\label{structure-sec}

The starting point of the argument of \cite{sahutoglu-straube} is the following local structure lemma for pseudoconvex domains whose boundary contains a $q$-dimensional complex manifold (Lemma $4.22$ of \cite{straube-book}). 

Here and in what follows we employ the notation $\Re(z)$ and $\Im(z)$ for the real and imaginary parts of a complex variable $z$, and $\partial_z$ (resp. $\partial_{\overline{z}}$) for $\frac{\partial_{\Re(z)}-i\partial_{\Im(z)}}{2}$ (resp.  $\frac{\partial_{\Re(z)}+i\partial_{\Im(z)}}{2}$).

\begin{lem}\label{bedford-fornaess-lemma}
Let $\Omega\subseteq\C^n$ be a smooth pseudoconvex domain with a $q$-dimensional complex manifold $M$ in its boundary. If $P\in M$ there exists a neighborhood $V$ of $P$ in $\C^n$ and a biholomorphism $\Phi:V\rightarrow\Phi(V)\subseteq\C^n$ such that: \begin{enumerate}
\item[(i)] $\Phi(P)=0$,
\item[(ii)] $\Phi(M\cap V)=\Phi(V)\cap \{z_{q+1}=\cdots=z_n=0\}$,
\item[(iii)] the exterior normal to $\Phi(\Omega\cap V)$ at points of $\Phi(V)\cap \{z_{q+1}=\cdots=z_n=0\}$ is given by the negative $\Re(z_n)$-axis. 
\end{enumerate}
\end{lem}

Before going further, let us rephrase the lemma in a somehow more explicit fashion in the $n=3$, $q=1$ case. Given $\Omega$ and $P$ as above there exist:
\begin{enumerate}
\item $t_1>0$, a disc $D_1\subseteq\C$ and a ball $B_1\subseteq\C\times\R$, both centered at $0$,
\item a smooth function $\varphi_1:D_1\times B_1\rightarrow(-t_1,t_1)$ such that $\varphi_1$ and its gradient $\nabla\varphi_1$ vanish identically on $D_1\times\{0\}$,
\item a neighborhood $V$ of $P$ in $\C^3$ and a biholomorphism $\Phi_1$ of $V$ onto the set $U_1:=\{z\in\C^3\colon z_1\in D_1,\ (z_2,\Im(z_3))\in B_1,\ |\Re(z_3)|<t_1\}$ such that $\Phi_1(P)=0$ and \be
\Omega_1:=\Phi_1(\Omega\cap V)=\{z\in U_1 \colon \Re(z_3)>\varphi_1(z_1,z_2,\Im(z_3))\}.
\ee 
\end{enumerate}

In other words, in the appropriate local coordinates the domain is fibered over a disc and every fiber is a pseudoconvex domain in $\C^2$ ``in normal form up to order $1$'' (i.e., condition (2) above holds). Our next lemma improves this result, allowing to put simultaneously all the fibers in normal form up to higher order, under our finite regular $2$-type condition.

\begin{lem}\label{structure-lemma}
Let $\Omega\subseteq \C^3$ be a smooth pseudoconvex domain with a one-dimensional complex manifold $M$ in its boundary. Assume moreover that $b\Omega$ has bounded order of contact with $2$-dimensional complex manifolds at a point $P\in M$. Then there exist:\begin{itemize}
\item a neighborhood $V$ of $P$ in $\C^3$, 
\item $t>0$, a disc $D\subseteq\C$ and a ball $B\subseteq\C\times\R$, both centered at $0$, 
\item a biholomorphism $\Phi:V\longrightarrow U$, where \be U:=\{z\in\C^3\colon z_1\in D, (z_2,\Im(z_3))\in B, |\Re(z_3)|<t\},\ee
such that $\Phi(P)=0$,
\item a smooth function $\varphi:D\times B\rightarrow(-t,t)$ and an even $\tau\in \N\setminus\{0\}$,
\end{itemize}
such that \begin{enumerate}
\item[(i)] $\Phi(\Omega\cap V)=\left\{z\in U\colon \Re(z_3)>\varphi(z_1,z_2,\Im(z_3))\right\}$,
\item[(ii)] $\partial_{\Im(z_3)}\varphi(z_1,0,0)=0$ for every $z_1\in D$,
\item[(iii)] $\partial^j_{z_2}\partial^k_{\overline{z}_2}\varphi(z_1,0,0)=0$ for every $z_1\in D$ and $j,k$ such that $j+k< \tau$,
\item[(iv)] there exist $z^*\in D$ and $j_0,k_0\geq1$ such that $j_0+k_0=\tau$ and\be\partial^{j_0}_{z_2}\partial^{k_0}_{\overline{z}_2}\varphi(z^*,0,0)\neq0.\ee
\end{enumerate}
\end{lem}

The proof of Lemma \ref{structure-lemma} is based on an induction procedure, whose inductive step we isolate as a separate lemma.

\begin{lem}\label{inductive-lemma}
For $m\geq1$ let $\varphi_m:D_m\times B_m\rightarrow(-t_m,t_m)$ be a smooth function, where $t_m>0$, $D_m\subseteq\C$ is a disc and $B_m\subseteq\C\times\R$ is a ball, both centered at $0$. 

Set $U_m:=\{z\in\C^3\colon z_1\in D_m,\ (z_2,\Im(z_3))\in B_m,\ |\Re(z_3)|<t_m\}$ and assume that the domain $\Omega_m:=\{z\in U_m\colon \Re(z_3)>\varphi_m(z_1,z_2,\Im(z_3))\}$ is pseudoconvex at boundary points where $\Re(z_3)=\varphi_m(z_1,z_2,\Im(z_3))$. Assume also that
\bel\label{der0m}
\partial^j_{z_2}\partial^k_{\overline{z}_2}\varphi_m(z_1,0,0)=0\qquad\forall z_1\in D_m,\quad \forall j,k\colon j+k\leq m,
\eel 
\bel\label{derIm}
\partial_{\Im(z_3)}\varphi_m(z_1,0,0)=0 \qquad \forall z_1\in D_m
\eel
and that \bel\label{der0m+1}
\partial^j_{z_2}\partial^k_{\overline{z}_2}\varphi_m(z_1,0,0)=0\qquad\forall z_1\in D_m,\quad \forall j,k\colon j,k\geq1,\quad j+k= m+1.
\eel
If $m$ is even, \eqref{der0m+1} is a consequence of \eqref{der0m} and \eqref{derIm}. 

Then there exist an open set $U_m'\subseteq U_m$, $t_{m+1}>0$, a disc $D_{m+1}\subseteq\C$ and a ball $B_{m+1}\subseteq\C\times\R$, both centered at $0$, a smooth function \be\varphi_{m+1}:D_{m+1}\times B_{m+1}\rightarrow(-t_{m+1},t_{m+1}),\ee and a biholomorphism \be
\Phi_{m+1}: U_m'\longrightarrow U_{m+1}:=\{z\in\C^3\colon z_1\in D_{m+1},\ (z_2,\Im(z_3))\in B_{m+1},\ |\Re(z_3)|<t_{m+1}\}
\ee such that the following holds:\begin{enumerate}
\item[(i)] $\Phi_{m+1}(\Omega_m\cap U'_m)=\Omega_{m+1}:=\{z\in U_{m+1}\colon \Re(z_3)>\varphi_{m+1}(z_1,z_2,\Im(z_3))\}$,
\item[(ii)] $\partial^j_{z_2}\partial^k_{\overline{z}_2}\varphi_{m+1}(z_1,0,0)=0\qquad\forall z_1\in D_{m+1},\quad \forall j,k\colon j+k\leq m+1$,
\item[(iii)] $\partial_{\Im(z_3)}\varphi_{m+1}(z_1,0,0)=0 \qquad \forall z_1\in D_{m+1}$.
\end{enumerate}
\end{lem}

\begin{proof}
Consider the biholomorphic transformations\be
T_{\delta}(z_1,z_2,z_3):=(z_1,\delta z_2,\delta^{m+1} z_3) \qquad (\delta>0).
\ee 
Observe that \be
T_{\delta^{-1}}(\Omega_m)=\{z\in T_{\delta^{-1}}(U_m)\colon \Re(z_3)>\delta^{-m-1}\varphi_m(z_1,\delta z_2,\delta^{m+1}\Im(z_3))\}.
\ee 
We put $\varphi_{m,\delta}(z_1, z_2,\Im(z_3)):=\delta^{-m-1}\varphi_m(z_1,\delta z_2,\delta^{m+1}\Im(z_3))$. Since $T_{\delta^{-1}}$ is biholomorphic, $T_{\delta^{-1}}(\Omega_m)$ is pseudoconvex at points where $\Re(z_3)=\varphi_{m,\delta}(z_1, z_2,\Im(z_3))$. 

We now apply Proposition \ref{aniso-prop} of Appendix \ref{aniso-app} with $f=\varphi_m$ and $\{D_\delta\}_{\delta>0}$, the dilations of $\C\times\R\equiv\R^3$ corresponding to the exponents $d_1=d_2=1$ and $d_3=m+1$, and $\ell=m+1$. Notice that $\beta\cdot d=\beta_1+\beta_2+(m+1)\beta_3<m+1$ means $\beta_1+\beta_2\leq m$ and $\beta_3=0$, so that the hypothesis of Proposition \ref{aniso-prop} are equivalent to hypothesis \eqref{der0m} above. We therefore have that $\varphi_{m,\delta}(z_1, z_2,\Im(z_3))$ converges pointwise, together with all its derivatives, to \bee
\varphi_{m,0}(z_1,z_2,\Im(z_3))&=&\sum_{\gamma\in\N^3\colon d\cdot\gamma=m+1}\frac{1}{\gamma!}\partial^{\gamma_1}_{z_2}\partial^{\gamma_2}_{\overline{z}_2}\partial^{\gamma_3}_{\Im(z_3)} \varphi_m(z_1,0,0)\cdot z_2^{\gamma_1}\overline{z}_2^{\gamma_2}\Im(z_3)^{\gamma_3}\\
&=&\sum_{\gamma_1+\gamma_2=m+1}\frac{1}{\gamma_1!\gamma_2!}\partial^{\gamma_1}_{z_2}\partial^{\gamma_2}_{\overline{z}_2}\varphi_m(z_1,0,0)\cdot z_2^{\gamma_1}\overline{z}_2^{\gamma_2}
\eee
In the second line we used the fact that $d_3=m+1$ and hypothesis \eqref{derIm}.

We can now invoke part \emph{(ii)} of Proposition \ref{pscvx-prop} of Appendix \ref{pscvx-app} to conclude that the domain \be
\left\{z\in D_m\times \C^2\colon \Re(z_3)>\sum_{\gamma_1+\gamma_2=m+1}\frac{1}{\gamma_1!\gamma_2!}\partial^{\gamma_1}_{z_2}\partial^{\gamma_2}_{\overline{z}_2}\varphi_m(z_1,0,0)\cdot z_2^{\gamma_1}\overline{z}_2^{\gamma_2}\right\}
\ee is pseudoconvex at boundary points in the interior of $D\times\C^2$. Equivalently, by part \emph{(i)} of the same proposition, \be
\sum_{\gamma_1+\gamma_2=m+1}\frac{1}{\gamma_1!\gamma_2!}\partial^{\gamma_1}_{z_2}\partial^{\gamma_2}_{\overline{z}_2}\varphi_m(z_1,0,0)\cdot z_2^{\gamma_1}\overline{z}_2^{\gamma_2} \text{ is plurisubharmonic}. 
\ee
Assume momentarily that $m$ is even. Then subharmonicity of the above function in the $z_2$ variable gives \be
\sum_{\gamma_1, \gamma_2\geq 1,\ \gamma_1+\gamma_2=m+1}\frac{1}{(\gamma_1-1)!(\gamma_2-2)!}\partial^{\gamma_1}_{z_2}\partial^{\gamma_2}_{\overline{z}_2}\varphi_m(z_1,0,0)\cdot z_2^{\gamma_1-1}\overline{z}_2^{\gamma_2-1}\geq 0.
\ee
Since a nonnegative homogeneous polynomial of odd degree is identically $0$, we conclude that $\partial^{\gamma_1}_{z_2}\partial^{\gamma_2}_{\overline{z}_2}\varphi_m(z_1,0,0)$ vanishes for every $z_1\in D_m$ and $\gamma_1,\gamma_2\geq1$ such that $\gamma_1+\gamma_2=m+1$. This proves that hypothesis \eqref{der0m+1} is superfluous when $m$ is even, as stated. Therefore, independently of the parity of $m$, we have that \bee
&&\sum_{\gamma_1+\gamma_2=m+1}\frac{1}{\gamma_1!\gamma_2!}\partial^{\gamma_1}_{z_2}\partial^{\gamma_2}_{\overline{z}_2}\varphi_m(z_1,0,0)\cdot z_2^{\gamma_1}\overline{z}_2^{\gamma_2}\\
&=&2\Re\left\{\frac{1}{(m+1)!}\partial^{m+1}_{z_2}\varphi_m(z_1,0,0)\cdot z_2^{m+1}\right\}
\eee is plurisubharmonic. An elementary computation then reveals that the determinant of the complex Hessian of the function above equals \be
-\left|\frac{1}{m!}\partial_{\overline{z}_1}\partial^{m+1}_{z_2}\varphi_m(z_1,0,0)\cdot z_2^m\right|^2.
\ee By plurisubharmonicity this quantity has to be nonnegative for every $z_1\in D_m$ and $z_2\in\C$, which is possible only if $\partial_{\overline{z}_1}\partial^{m+1}_{z_2}\varphi_m(z_1,0,0)$ vanishes on $D_m$, that is, if $\partial^{m+1}_{z_2}\varphi_m(\cdot,0,0)$ is holomorphic.

We are now ready to define $\Phi_{m+1}$ as follows:\be
\Phi_{m+1}(z_1,z_2,z_3):=\left(z_1,z_2,z_3-\frac{2}{(m+1)!}\partial^{m+1}_{z_2}\varphi_m(z_1,0,0)\cdot z_2^{m+1}\right).
\ee 
By what we just proved, it is clear that $\Phi_m$ is a biholomorphism of $U_m$ onto its image. Moreover, it is clear that there is a smaller neighborhood of $0$ $U'_m\subseteq U_m$ such that 
\be
\Phi_{m+1}(\Omega_m\cap U'_m)=\{z\in U_{m+1}\colon \Re(z_3)>\varphi_{m+1}(z_1,z_2,\Im(z_3))\},
\ee where $U_{m+1}$ is as in the statement and \bee
\varphi_{m+1}(z_1,z_2, \Im(z_3))&=&\varphi_m\left(z_1,z_2,\Im(z_3)+\frac{2}{(m+1)!}\Im\left(\partial^{m+1}_{z_2}\varphi_m(z_1,0,0)\cdot z_2^{m+1}\right)\right)\\
&-&\frac{2}{(m+1)!}\Re\left\{\partial^{m+1}_{z_2}\varphi_m(z_1,0,0)\cdot z_2^{m+1}\right\}.
\eee
An elementary application of the chain rule finally gives\bee
\partial_{\Im(z_3)}\varphi_{m+1}(z_1,0,0)&=&\partial_{\Im(z_3)}\varphi_m(z_1,0,0)=0,\\
\partial^j_{z_2}\partial^k_{\overline{z}_2}\varphi_{m+1}(z_1,0,0)&=&\partial^j_{z_2}\partial^k_{\overline{z}_2}\varphi_m(z_1,0,0)=0\qquad \forall j,k\colon j+k\leq m,\\
\partial^j_{z_2}\partial^k_{\overline{z}_2}\varphi_{m+1}(z_1,0,0)&=&\partial^j_{z_2}\partial^k_{\overline{z}_2}\varphi_m(z_1,0,0)=0\qquad \forall j,k\geq1\colon j+k= m+1,\\
\partial^{m+1}_{z_2}\varphi_{m+1}(z_1,0,0)&=&\partial^{m+1}_{z_2}\varphi_m(z_1,0,0)-\partial^{m+1}_{z_2}\varphi_m(z_1,0,0)=0,\\
\partial^{m+1}_{\overline{z}_2}\varphi_{m+1}(z_1,0,0)&=&\partial^{m+1}_{\overline{z}_2}\varphi_m(z_1,0,0)-\partial^{m+1}_{\overline{z}_2}\varphi_m(z_1,0,0)=0.
\eee
This concludes the proof.
\end{proof}

We are now ready to present the proof of Lemma \ref{structure-lemma}.

\begin{proof}[Proof of Lemma \ref{structure-lemma}]

We consider the following iteration scheme. \begin{enumerate}
\item We initially set $m=1$ (recall how $\Omega_1, D_1, B_1, U_1, \varphi_1$ and $\Phi_1$ have been defined just before the statement of Lemma \ref{structure-lemma}).
\item We are given $\Omega_m, D_m, B_m, U_m, \varphi_m$ and $\Phi_m$. If there exist $z^*\in D_m$ and $j_0,k_0\geq1$ such that $j_0+k_0=m+1$ and \be\partial^{j_0}_{z_2}\partial^{k_0}_{\overline{z}_2}\varphi_m(z^*,0,0)\neq0,\ee
we stop the iteration. If the opposite holds, i.e., \be
\partial^j_{z_2}\partial^k_{\overline{z}_2}\varphi_m(z_1,0,0)=0\qquad\forall z_1\in D_m,\quad \forall j,k\colon j,k\geq1,\quad j+k= m+1,
\ee go to step (3).
\item Lemma \ref{inductive-lemma} (whose hypothesis are easily seen to be satisfied) gives us \be
\Omega_{m+1}, D_{m+1}, B_{m+1}, U_{m+1}, \varphi_{m+1} \text{ and } \Phi_{m+1}.\ee Increase $m$ by $1$ and go back to step (2).
\end{enumerate}

We claim that this iteration scheme stops after finitely many steps. Let us first draw the consequences of the claim and then prove it.

If the iteration stops after $m$ steps, by Lemma \ref{inductive-lemma} we have in particular that $m$ has to be odd. The mapping $\Psi_m:=\Phi_1^{-1}\circ\cdots\circ\Phi_m^{-1}$ is well-defined on $U_m$ and it is a biholomorphism onto its image $V_m$ such that $\Psi_m(0)=P$ and\be
\Psi_m(\Omega_m)=\Omega\cap V_m,\quad \Omega_m:=\{z\in U_m\colon \Re(z_3)>\varphi_m(z_1,z_2,\Im(z_3))\},
\ee
where $\partial^j_{z_2}\partial^k_{\overline{z}_2}\varphi_m(z_1,0,0)=0$ for every $z_1\in D_m$ and every $j,k$ such that $j+k\leq m$, but $\partial^{j_0}_{z_2}\partial^{k_0}_{\overline{z}_2}\varphi_m(z^*,0,0)\neq0$ for at least a point $z^*\in D_m$ and a couple of exponents $j_0,k_0\geq1$ such that $j_0+k_0=m+1$.

Setting $t=t_m, D=D_m, B=B_m, U=U_m, V=V_m, \varphi:=\varphi_m, \Phi=\Psi_m^{-1}$ and $\tau=m+1$, we have the thesis.

We are left with proving that the iteration actually halts. We argue by contradiction, assuming that it never stops. In this case, we can define $\Psi_m$ and $V_m$ as above for every $m$ and consider the composition\be
H_m:=U_m\cap \C^2\times \{0\}\hookrightarrow U_m\stackrel{\Psi_m}{\longrightarrow} \C^3,
\ee that is a holomorphic mapping with Jacobian of rank 2 (because $\Psi_m$ is a biholomorphism) such that $H_m(0,0)=P$. Observe that $\rho_m\circ\Psi_m^{-1}$ is a local defining function for $\Omega$ on $V_m$, if \be
\rho_m(z):=\varphi_m(z_1,z_2,\Im(z_3))-\Re(z_3), 
\ee and that\be
\rho_m\circ\Psi_m^{-1}\circ H_m(z_1,z_2) = \varphi_m(z_1,z_2,0)=\OO(|z_2|^{m+1}).
\ee
Since $m$ is arbitrary, this contradicts the finite-type hypothesis and completes the proof.

\end{proof}

\section{A non-compactness lemma}\label{noncpt-sec}

Special instances of the following lemma play a key role both in \cite{fu-straube-conv} and in \cite{sahutoglu-straube}, although it does not appear as a separate statement in these papers. The strictly pseudoconvex version of the result can be found in \cite{straube-book} (Lemma $4.23$). We will need a more general statement, whose short proof is obtained adapting the arguments of the cited papers.

\begin{lem}\label{noncpt-lemma}
Let $\Omega_2\subseteq\Omega_1\subseteq\C^n$ be domains that share a boundary point $P$. Denote by $K_1$ (resp. $K_2$) the Bergman kernel of $\Omega_1$ (resp. $\Omega_2$) and assume that:
\begin{enumerate}
\item[(i)] there exists a neighborhood $V$ of $P$ in $\C^n$ such that $K_1(z,\cdot)$ is bounded on $\Omega_1\cap V$ for every $z\in \Omega_1\cap V$,
\item[(ii)] there exists a sequence $\{P_n\}_n\subseteq\Omega_2$ converging to $P$ such that \be
\lim_{n\rightarrow+\infty}K_1(P_n,P_n)=+\infty\quad\text{and}\quad K_2(P_n,P_n)\lesssim K_1(P_n,P_n).
\ee 
\end{enumerate}
Then the restriction operator $A(\Omega_1)\rightarrow A(\Omega_2)$ is not compact. 
\end{lem}
The symbol $\lesssim$ means that the implicit constant is independent of $n$.

Before going to the proof, a few comments are in order.

\begin{enumerate}

\item[(a)] An effective way to verify assumption \emph{(i)} is via the observation (see \cite{kerzman}) that, since $K_1$ is holomorphic in its first variable,\be
K_1(z,w)=\int_{\Omega_1}K_1(z',w)\chi_z(z')d\lambda(z')=\overline{B_1(\chi_z)(w)}\qquad\forall z,w\in\Omega_1.
\ee Here $\chi_z$ is an $L^1$-normalized real-valued radial test function centered at $z$ and with support inside $\Omega_1$, while $B_1$ is the Bergman projection of $\Omega_1$. Therefore, all we need for \emph{(i)} to hold is that $B_1$ maps test functions on $\Omega_1\cap V$ to functions that are bounded on the same region, a rather weak regularity requirement, which is certainly satisfied if $\Omega_1$ is pseudoconvex, smooth in a neighborhood of $P$, and has a compact $\dbar$-Neumann operator. In fact under these assumptions a well-known theorem of Kohn and Nirenberg states that $N_1$, the $\dbar$-Neumann operator, maps smooth $(0,1)$-forms to forms that are smooth up to the boundary in a neighborhood of $P$ (see Theorem 4.6 of \cite{straube-book}, where the assumption is that the domain is everywhere smooth, but the argument is local). Combining this with Kohn's identity $B_1(\chi_z)=\chi_z-\dbar^*N_1\dbar \chi_z$, the conclusion follows immediately.

\item[(b)] The condition $\lim_{n\rightarrow+\infty}K_1(P_n,P_n)=+\infty$ holds for any sequence $\{P_n\}_n$ approaching a boundary point of $\Omega_1$ that satisfies the \emph{outer cone condition} (see \cite{jarnicki-pflug}), which is the case if $\Omega_1$ is smooth in a neighborhood of $P$.

\item[(c)] The inequality $K_2(z,z)\geq K_1(z,z)$ follows from the inclusion $\Omega_2\subseteq\Omega_1$ and the variational characterization \bel\label{var-bergman}
K_j(z,z):=\sup\left\{|f(z)|^2\colon f\in A(\Omega_j), \quad ||f||_{L^2}=1\right\}\qquad (j=1,2).
\eel

\end{enumerate}

Thus the key assumption of Lemma \ref{noncpt-lemma} is the bound $K_2(P_n,P_n)\lesssim K_1(P_n,P_n)$. Observe that in view of \eqref{var-bergman}, the larger $\Omega_2$ is, the easier it is for it to satisfy this condition.

\begin{proof}
We argue by contradiction, assuming that the restriction operator is compact. Consider the sequence of functions \be f_n(z):=\frac{K_1(z,P_n)}{\sqrt{K_1(P_n,P_n)}}\qquad (z\in\Omega_1).\ee
By elementary properties of the Bergman kernels, they are holomorphic and satisfy $\int_{\Omega_1}|f_n|^2=1$. Using our assumption and passing to a subsequence, we can assume that $f_n$ converges in $L^2(\Omega_2)$ and pointwise almost everywhere to a function $g\in A(\Omega_2)$. 

Moreover, by \emph{(i)} and the first half of \emph{(ii)}, we immediately have that 
\be\lim_{n\rightarrow+\infty}f_n(z)=0 \qquad \forall z\in V\cap\Omega_1.\ee Hence $g$ vanishes almost everywhere on $V\cap\Omega_2$. Being holomorphic, it vanishes everywhere on $\Omega_2$.

Now consider $K_1(\cdot,P_n)$ as an element of $A(\Omega_2)$. The reproducing formula for $\Omega_2$ gives\bee
K_1(P_n,P_n)&=&\int_{\Omega_2}K_2(P_n,w)K_1(w,P_n)\\
&\leq& \sqrt{\int_{\Omega_2}|K_2(P_n,\cdot)|^2}\cdot\sqrt{\int_{\Omega_2}|K_1(\cdot,P_n)|^2}\\
&=&\sqrt{K_2(P_n,P_n)}\cdot\sqrt{\int_{\Omega_2}|K_1(\cdot,P_n)|^2}\\
&\lesssim& \sqrt{K_1(P_n,P_n)}\cdot\sqrt{\int_{\Omega_2}|K_1(\cdot,P_n)|^2},
\eee by Cauchy-Schwartz and the second half of \emph{(2)}. Rearranging this inequality, we find $\int_{\Omega_2}|f_n|^2\gtrsim 1$, which clearly contradicts the fact that $f_n$ converges to $0$ in $L^2(\Omega_2)$.  
\end{proof}

\section{Bergman kernels on the diagonal on domains of finite type in $\C^2$}\label{bergman-sec}

We recall a result of Catlin about Bergman kernels of smooth pseudoconvex domains of finite-type in $\C^2$. 

\begin{thm}\label{catlin-thm} Let $\Omega\subseteq\C^2$ be a smooth pseudoconvex domain.  Assume that $P\in b\Omega$ is a point of type $\tau<+\infty$ and denote by $\nu$ the exterior normal to $\Omega$ at $P$. Let $V$ be a bounded neighborhood of $P$ in $\C^2$ and $K$ the Bergman kernel of $\Omega\cap V$. Then there exists $\delta_0>0$ such that\be
K(P-\delta\nu,P-\delta\nu)\approx \delta^{-2-\frac{2}{\tau}}\qquad\forall \delta\in(0,\delta_0).
\ee
\end{thm}

This is Theorem $2$ of \cite{catlin-invariant}, except that it deals with the Bergman kernel of a bounded smooth pseudoconvex domain. To obtain the result stated above, one can apply the localization lemma of Ohsawa \cite{ohsawa-localization}.

Next, we want to state an elementary upper bound for the Bergman kernel of a non-smooth model domain ``of type $\tau$'', that will be needed later. The proof is standard.

\begin{prop}\label{bergman-model-prop}
Consider the domain\be
\Omega_{\tau,C}:=\left\{(z_1,z_2)\in\C^2\colon \Re(z_2)>C|z_1|^\tau+C|\Im(z_2)|\right\},
\ee where $\tau, C>0$. Let $W$ be a neighborhood of the origin in $\C^2$ and $K$ be the Bergman kernel of $\Omega_{\tau,C}\cap W$ . Then there exists $\delta_0>0$ such that \be K((0,\delta),(0,\delta))\lesssim \delta^{-2-\frac{2}{\tau}}\qquad \forall\delta\in (0,\delta_0).\ee
\end{prop}

\begin{proof}
Let us check that the bidisc $E_\delta:=D(0,\delta^{\frac{1}{\tau}})\times D((2C+1)\delta,\delta)$ is contained in $\Omega_{m,C}$ for every $\delta\in(0,1)$. If $(z_1,z_2)\in E_\delta$, $|\Im(z_2)|<\delta$ and \be
C|z_1|^\tau+C|\Im(z_2)|<2C\delta.
\ee Since $\Re(z_2)>2C\delta$, if $\delta<\delta_0$ (with $\delta_0$ dependent on $W$) we have $(z_1,z_2)\in \Omega_{\tau,C}$. 

Let now $f$ be a holomorphic function on $\Omega_{\tau,C}\cap W$ having $L^2$ norm equal to $1$. Since holomorphic functions are pluriharmonic, we can estimate (for $\delta<\delta_0$ depending on $W$)\be
|f(0,(2C+1)\delta)|^2=\left|\frac{1}{\text{Vol}(E_\delta)}\int_{E_\delta}f\right|^2\leq \frac{1}{\text{Vol}(E_\delta)}\int_{E_\delta}|f|^2\leq \frac{1}{\text{Vol}(E_\delta)},
\ee where we used the mean value property in both coordinates and Jensen's inequality. Taking the supremum over $f\in A(\Omega_{\tau,C}\cap W)$ of norm $1$ and computing the volume we conclude that \be
K((0,(2C+1)\delta),(0,(2C+1)\delta))\leq \frac{1}{4\pi^2\delta^{2+\frac{2}{\tau}}}\qquad\forall \delta<\delta_0.
\ee Rescaling the variable $\delta$, we obtain the stated estimate.
\end{proof}

\section{Proof of Theorem \ref{main-thm}}\label{proof-sec}

We argue by contradiction, assuming that $N_1$ is compact.

We begin by applying Lemma \ref{structure-lemma} and keeping the notation of its statement. We may assume that the neighborhood $V$ of the point $P$ is strictly pseudo-convex (for example, a ball), if we replace $U$ with a smaller open set. 

An anisotropic Taylor expansion (Proposition \ref{taylor-prop} of Appendix \ref{aniso-app}), together with the conclusions \emph{(ii)} and \emph{(iii)} of Lemma \ref{structure-lemma}, gives 
\bee
\varphi(z_1,z_2,\Im(z_3))&=&\sum_{\gamma_1+\gamma_2=\tau}\frac{\partial^{\gamma_1}_{z_2}\partial^{\gamma_2}_{\overline{z}_2}\varphi(z_1,0,0)}{\gamma_1!\gamma_2!}\cdot z_2^{\gamma_1}\overline{z}_2^{\gamma_2}  + \OO(|z_2|^{\tau+1} + |\Im(z_3)|)\\
&\leq&C|z_2|^\tau + C|\Im(z_3)|,
\eee
where $C$ and the implicit constant of the big O are uniform in $z_1\in D$. Thus\bel\label{inclusion}
D'\times \Omega_2\subseteq \widetilde{\Omega}, 
\eel where $D'\subseteq D$ is a small disc centered at $0$, \be\Omega_2=\{(z_2,z_3)\colon \Re(z_3)>C|z_2|^\tau + C|\Im(z_3)|, \quad |z_2|^2+|z_3|^2<\eps^2\},\ee for some $\eps>0$, and 
\be\widetilde\Omega:=\left\{z\in U\colon \Re(z_3)>\varphi(z_1,z_2,\Im(z_3))\right\}.\ee

We claim that the two-dimensional domains \be
\Omega_1:=\left\{(z_2,z_3)\in\C^2\colon (z^*,z_2,z_3)\in \widetilde\Omega\right\}\qquad\text{and}\qquad \Omega_2
\ee 
satisfy the hypothesis of Lemma \ref{noncpt-lemma}. Here $z^*$ is as in the statement of Lemma \ref{structure-lemma}, and $P=(z^*,0,0)$ is the common boundary point. Conditions $(ii)$, $(iii)$, and $(iv)$ of Lemma \ref{structure-lemma} imply that $P$ is point of type $\tau$ of $\Omega_1$, and that the exterior normal at $P$ is $\nu=(0,-1)$. Thus Theorem \ref{catlin-thm} and Proposition \ref{bergman-model-prop} combine to give the key hypothesis (the second half of $(ii)$ of Lemma \ref{noncpt-lemma}). For the other assumptions recall that, by our observations after the statement of the lemma, all we need is the compactness of the $\dbar$-Neumann operator of $\Omega_1$. To see this, notice that this domain is the intersection of a $2$-dimensional smooth bounded pseudoconvex domain of finite-type with $\{(z_2,z_3)\in\C^2\colon (z^*,z_2,z_3)\in U\}$. The latter set is a slice of $\Phi(V)$, which is strictly pseudoconvex thanks to our choice of $V$. Since compactness localizes to such intersections (see part (2) of Proposition $4.4$ in \cite{straube-book}), our goal follows from a theorem of Catlin (see \cite{catlin-global}).

Now the proof follows closely the arguments of \cite{fu-straube-conv} and \cite{sahutoglu-straube}. Lemma \ref{noncpt-lemma} yields a bounded sequence $\{f_n\}_n\subseteq A(\Omega_1)$ such that no subsequence converges in $L^2(\Omega_2)$. The Ohsawa-Takegoshi extension theorem (see \cite{ohsawa-takegoshi}, or Theorem $2.17$ of \cite{straube-book}) gives $\{F_n\}_n\subseteq A(\widetilde\Omega)$ such that \be
F_{n|\Omega_1}=f_n\quad\forall n\quad\text{and}\quad ||F_n||_{L^2(\widetilde\Omega)}\lesssim 1.
\ee

The $(0,1)$-forms $F_nd\overline{z}_1$ are closed  and uniformly bounded in $L^2(\widetilde\Omega)$, and so are their pull-backs $\Phi^*(F_nd\overline{z}_1)$ (the biholomorphism $\Phi:V\rightarrow U'$ may be assumed to be smooth up to the boundary to guarantee this), which are defined on $\Omega\cap V$. The latter is a domain (again taking $V$ small enough) whose $\dbar$-Neumann operator on $(0,1)$-forms is compact. In fact we assumed, by contradiction, that this is the case for the whole domain $\Omega$, and we can localize thanks to Proposition $4.4$ of \cite{straube-book} as before, because $V$ is strictly pseudoconvex. It is an elementary fact that the compactness of the $\dbar$-Neumann operator implies the compactness of the canonical solution operator (see, e.g., Proposition $4.2$ of \cite{straube-book}), so that composing with $\Phi^{-1}$ and passing to a subsequence we get a Cauchy sequence $\{g_n\}_n\subseteq L^2(\widetilde\Omega)$ with the property that $\dbar g_n=F_nd\overline{z}_1$. In particular $\partial_{\overline{z_1}}g_n=F_n$. 

If $(z_2,z_3)\in \Omega_2$ and $\eta$ is a radial real-valued test function on $D'$ such that $\int\eta=1$, we have the bound \bee
&&\int_{\Omega_2}|f_n(z^*,z_2,z_3)-f_m(z^*,z_2,z_3)|^2\\
&=&\int_{\Omega_2}\left|\int_{D'} \left(F_n(z_1,z_2,z_3)-F_m(z_1,z_2,z_3)\right)\eta(z_1)d\lambda(z_1)\right|^2\\
&=&\int_{\Omega_2}\left|\int_{D'} \left(g_n(z_1,z_2,z_3)-g_m(z_1,z_2,z_3)\right)\partial_{\overline{z_1}}\eta(z_1)d\lambda(z_1)\right|^2\\
&\lesssim& \int_{\Omega_2}\int_{D'} |g_n(z_1,z_2,z_3)-g_m(z_1,z_2,z_3)|^2d\lambda\\
&\leq& \int_{\widetilde\Omega}|g_n(z_1,z_2,z_3)-g_m(z_1,z_2,z_3)|^2d\lambda,
\eee
where in the second line we used the inclusion \eqref{inclusion} and the mean-value property for holomorphic functions, and in the third line an integration by parts. This proves that $\{f_n\}_n$ is Cauchy in $L^2(\Omega_2)$, the contradiction we seeked. Theorem \ref{main-thm} is proved.

\appendix 

\section{Anisotropic scaling}\label{aniso-app}

Given $d\in(\N\setminus\{0\})^q$, we consider the associated one-parameter group of anisotropic dilations of $\R^q$\be
D_\delta y:=(\delta^{d_1}y_1,\dots,\delta^{d_q}y_q) \qquad (\delta>0).
\ee
It is useful to introduce a corresponding anisotropic norm $N(y):=\max_{j=1,\dots,q}|y_j|^{\frac{1}{d_j}}$, that satisfies $N(D_\delta y)=\delta N(y)$ ($\forall\delta>0$). Our first elementary technical tool is an anisotropic Taylor series.

\begin{prop}\label{taylor-prop}
If $g$ is a smooth function defined in a neighborhood of $0$ in $\R^q$,
\be
g(y)=\sum_{\gamma\colon d\cdot \gamma\leq \ell}\frac{1}{\gamma!}\partial_y^\gamma g(0)y^\gamma+ \OO(N(y)^{\ell+1}),
\ee
where the constant implicit in the big O depends on an upper bound on finitely many derivatives of $g$ in a neighborhood of $0$.
\end{prop}

\begin{proof}
The ordinary Taylor expansion is\be
g(y)=\sum_{\gamma\colon |\gamma|\leq \ell}\frac{1}{\gamma!}\partial_y^\gamma g(0)y^\gamma+\OO(|y|^{\ell+1}),
\ee where $|\gamma|=\sum_{j=1}^q\gamma_j$ and the constant implicit in the big O depends on an upper bound on finitely many derivatives of $g$ in a neighborhood of $0$. Hence, 
\bee
g(y)&=&\sum_{\gamma\colon d\cdot \gamma\leq \ell}\frac{1}{\gamma!}\partial_y^\gamma g(0)y^\gamma+\sum_{\gamma\colon d\cdot \gamma\geq \ell+1, |\gamma|\leq \ell}\frac{1}{\gamma!}\partial_y^\gamma g(0)y^\gamma+ \OO(|y|^{\ell+1})\\
&=&\sum_{\gamma\colon d\cdot \gamma\leq \ell}\frac{1}{\gamma!}\partial_y^\gamma g(0)y^\gamma+\sum_{\gamma\colon d\cdot \gamma\geq \ell+1, |\gamma|\leq \ell}\OO(N(y)^{d\cdot\gamma})+ \OO(|y|^{\ell+1})\\
&=&\sum_{\gamma\colon d\cdot \gamma\leq \ell}\frac{1}{\gamma!}\partial_y^\gamma g(0)y^\gamma+ \OO(N(y)^{\ell+1}),
\eee
where in the first identity we used the fact that $d\cdot \gamma$ takes only integer values.
\end{proof}

We can now state and prove a useful proposition about anisotropic scaling of functions.

\begin{prop}\label{aniso-prop}
Let $B_1\subseteq\R^p$ and $B_2\subseteq\R^q$ be open euclidean balls centered at $0$, and let $f:B_1\times B_2\rightarrow\R$ be a smooth function. Let $d\in(\N\setminus\{0\})^q$ and $\{D_{\delta}\}_{\delta>0}$ the associated one-parameter group of dilations.

If $\partial_y^\beta f(x,0)=0$ for every $x\in B_1$ and for every $\beta\in\N^q$ such that $d\cdot\beta<\ell$, then\bel\label{aniso}
\lim_{\delta\rightarrow0+}\delta^{-\ell}f(x,D_\delta y) = \sum_{\gamma\colon d\cdot\gamma=\ell}\frac{1}{\gamma!}\partial^\gamma_yf(x,0)y^\gamma\qquad\forall (x,y)\in B_1\times\R^q,
\eel
where we adopted the usual multi-index notation. Notice that the rescaled function in the left-hand side is defined on $B_1\times D_{\delta^{-1}}B_2$, so that the limit is defined for every $(x,y)\in B_1\times\R^q$.

Moreover, the pointwise convergence holds also for derivatives, i.e., setting \be 
f_\delta(x,y):=\delta^{-\ell}f(x,D_\delta y)\quad\text{and}\quad f_0(x,y):= \sum_{\gamma\colon d\cdot\gamma=\ell}\frac{1}{\gamma!}\partial^\gamma_yf(x,0)y^\gamma,
\ee 
we have
\bel\label{aniso-der}
\lim_{\delta\rightarrow0+}\partial_x^\alpha\partial_y^\beta f_\delta(x,y)=\partial_x^\alpha\partial_y^\beta f_0(x,y)\qquad\forall (x,y)\in B_1\times\R^q.
\eel
\end{prop}

\begin{proof}
Observe that $\partial_x^\alpha\partial_y^\beta f_\delta(x,y)=\delta^{-\ell+d\cdot\beta}\partial_x^\alpha\partial_y^\beta f(x,D_\delta y)$ and recall that by assumption $\partial_y^\gamma\partial_x^\alpha\partial_y^\beta f(x,0)=0$ for every $x$ and every $\gamma$ such that $d\cdot\gamma<\ell-d\cdot\beta$. Proposition \ref{taylor-prop} yields
\bee
\partial_x^\alpha\partial_y^\beta f_\delta(x,y)&=&\delta^{-\ell+d\cdot\beta}\left(\sum_{\gamma:d\cdot\gamma = \ell-d\cdot\beta}\frac{\delta^{d\cdot\gamma}}{\gamma!}\partial_x^\alpha\partial_y^{\beta+\gamma} f(x,0)y^\gamma + \delta^{\ell-d\cdot\beta+1}\OO(N(y)^{\ell-d\cdot\beta+1})\right)\\
&=&\sum_{\gamma:d\cdot\gamma = \ell-d\cdot\beta}\frac{1}{\gamma!}\partial_x^\alpha\partial_y^{\beta+\gamma} f(x,0)y^\gamma + \delta\OO(N(y)^{\ell-d\cdot\beta+1}).
\eee Notice that the first summand may be empty (when $d\cdot\beta>\ell$). It is easy to see that the first term above equals $\partial_x^\alpha\partial_y^\beta f_0(x,y)$. Letting $\delta$ tend to $0$, we get the thesis.
\end{proof}

\section{Pseudoconvexity and defining functions}\label{pscvx-app}

Fix a neighborhood $V$ of the origin in $\C^{n-1}\times \R$ and a constant $t>0$. We define\be
U:=\{z=(z',z_n)\in\C^n\colon (z',\Im{z_{n-1}})\in V, \ |\Re(z_n)|<t \}.
\ee

If $\varphi:V\rightarrow \R$ is a $C^2$ function, we consider the domain\be
\Omega:=\{z\in U\colon \Re(z_n)>\varphi(z',\Im(z_{n}))\}.
\ee We want to express in terms of $\varphi$ the condition of pseudoconvexity of $\Omega$ at points where $\Re(z_n)=\varphi(z',\Im(z_{n}))$. It is clear that $b\Omega\cap U$ is parametrized by 
\be
V\ni(z',y)\longmapsto(z',\varphi(z',y)+iy).
\ee The complex tangent $T^{(1,0)}_zb\Omega$ at the point corresponding to $(z',y)$ is the set of $u\in\C^n$ such that\be
\sum_{j=1}^{n-1}\partial_{z_j}\varphi(z',y)u_j = \frac{1}{2}\left(1+i\partial_y\varphi(z',y)\right)u_n,
\ee as one can compute using the defining function $\rho(z)=\varphi(z',\Im(z_{n}))-\Re(z_n)$. 

Since $1+i\partial_y\varphi(z',y)$ does not vanish on $V$, we can parametrize $T^{(1,0)}_zb\Omega$ by \be
\C^{n-1}\ni u'=(u_1,\dots,u_{n-1})\longmapsto \left(u', \frac{\sum_{j=1}^{n-1}\partial_{z_j}\varphi(z',y)u_j}{\frac{1}{2}\left(1+i\partial_y\varphi(z',y)\right)}\right).
\ee 
Computing the second order derivatives of $\rho$ we can finally express the pseudoconvexity of $b\Omega$ on $U$ as the inequality\bee
&&\sum_{j,k=1}^{n-1}\partial^2_{z_j\overline{z}_k}\varphi(z',y) \cdot u_j\overline{u}_k + \partial^2_y\varphi(z',y)\left|\frac{\sum_{k=1}^{n-1}\partial_{z_k}\varphi(z',y)\cdot u_k}{1+i\partial_y\varphi(z',y)}\right|^2\\
&+&2\Re\left\{\sum_{j=1}^{n-1}i\partial^2_{z_j y}\varphi(z',y) \cdot u_j\cdot \overline{\frac{\sum_{k=1}^{n-1}\partial_{z_k}\varphi(z',y)\cdot u_k}{1+i\partial_y\varphi(z',y)}}\right\}\geq 0,
\eee that must hold for every $u\in\C^{n-1}$ and $(z',y)\in V$. 

From this inequality one can easily deduce the following proposition.

\begin{prop}\label{pscvx-prop}
\begin{enumerate}

\item[(i)] If $\varphi(z',y)\equiv\varphi(z')$ is independent of $y$, the pseudoconvexity of $\Omega$ is equivalent to the plurisubharmonicity of $\varphi(z')$.

\item[(ii)] If $\{\varphi_k\}_{k\in\N}$ is a sequence of $C^2$ real-valued functions as above such that the corresponding domains $\Omega_k$ are all pseudoconvex (at points where $\Re(z_n)=\varphi_k$), and if $\varphi_k$ converges pointwise together with all the derivatives up to order $2$ to a $C^2$ function $\varphi$, then the domain $\Omega$ associated to $\varphi$ is also pseudoconvex (at points where $\Re(z_n)=\varphi$).
\end{enumerate}

\end{prop}

\section*{Acknowledgements}

The author would like to thank the people of the complex analysis group of the University of Vienna for their comments and observations on his work, in particular Friedrich Haslinger, Bernhard Lamel, and Michael Reiter. The author's research have been supported by the FWF-project P28154.

\bibliographystyle{amsalpha}
\bibliography{noncompactness-C3}

\providecommand{\bysame}{\leavevmode\hbox to3em{\hrulefill}\thinspace}
\providecommand{\MR}{\relax\ifhmode\unskip\space\fi MR }
\providecommand{\MRhref}[2]{%
  \href{http://www.ams.org/mathscinet-getitem?mr=#1}{#2}
}
\providecommand{\href}[2]{#2}
\begin{thebibliography}{BER99}

\bibitem[BER99]{ber-book}
M.~Salah Baouendi, Peter Ebenfelt, and Linda~Preiss Rothschild, \emph{Real
  submanifolds in complex space and their mappings}, Princeton Mathematical
  Series, vol.~47, Princeton University Press, Princeton, NJ, 1999.
  \MR{1668103}

\bibitem[BFs81]{bedford-fornaess}
Eric Bedford and J.~E. Forn\ae~ss, \emph{Complex manifolds in pseudoconvex
  boundaries}, Duke Math. J. \textbf{48} (1981), no.~1, 279--288. \MR{610187}

\bibitem[Cat84]{catlin-global}
David~W. Catlin, \emph{Global regularity of the {$\bar \partial $}-{N}eumann
  problem}, Complex analysis of several variables ({M}adison, {W}is., 1982),
  Proc. Sympos. Pure Math., vol.~41, Amer. Math. Soc., Providence, RI, 1984,
  pp.~39--49. \MR{740870}

\bibitem[Cat89]{catlin-invariant}
\bysame, \emph{Estimates of invariant metrics on pseudoconvex domains of
  dimension two}, Math. Z. \textbf{200} (1989), no.~3, 429--466. \MR{978601}

\bibitem[CD06]{charpentier-dupain-bergman}
Philippe Charpentier and Yves Dupain, \emph{Estimates for the {B}ergman and
  {S}zeg{\"o} projections for pseudoconvex domains of finite type with locally
  diagonalizable {L}evi form}, Publ. Mat. \textbf{50} (2006), no.~2, 413--446.
  \MR{2273668}

\bibitem[CF05]{christ-fu}
Michael Christ and Siqi Fu, \emph{Compactness in the
  {$\overline\partial$}-{N}eumann problem, magnetic {S}chr\"odinger operators,
  and the {A}haronov-{B}ohm effect}, Adv. Math. \textbf{197} (2005), no.~1,
  1--40. \MR{2166176 (2006i:32041)}

\bibitem[cSgS06]{sahutoglu-straube}
S{\"o}nmez \c~Sahuto\u~glu and Emil~J. Straube, \emph{Analytic discs,
  plurisubharmonic hulls, and non-compactness of the
  {$\overline\partial$}-{N}eumann operator}, Math. Ann. \textbf{334} (2006),
  no.~4, 809--820. \MR{2209258}

\bibitem[DK99]{dangelo-kohn}
John~P. D'Angelo and Joseph~J. Kohn, \emph{Subelliptic estimates and finite
  type}, Several complex variables ({B}erkeley, {CA}, 1995--1996), Math. Sci.
  Res. Inst. Publ., vol.~37, Cambridge Univ. Press, Cambridge, 1999,
  pp.~199--232. \MR{1748604}

\bibitem[FS98]{fu-straube-conv}
Siqi Fu and Emil~J. Straube, \emph{Compactness of the
  {$\overline\partial$}-{N}eumann problem on convex domains}, J. Funct. Anal.
  \textbf{159} (1998), no.~2, 629--641. \MR{1659575}

\bibitem[FS01]{fu-straube-comp}
\bysame, \emph{Compactness in the {$\overline\partial$}-{N}eumann problem},
  Complex analysis and geometry ({C}olumbus, {OH}, 1999), Ohio State Univ.
  Math. Res. Inst. Publ., vol.~9, de Gruyter, Berlin, 2001, pp.~141--160.
  \MR{1912737 (2004d:32053)}

\bibitem[Gre74]{greiner}
Peter Greiner, \emph{Subelliptic estimates for the {$\bar \delta $}-{N}eumann
  problem in {$C^{2}$}}, J. Differential Geometry \textbf{9} (1974), 239--250.
  \MR{0344702}

\bibitem[JP13]{jarnicki-pflug}
Marek Jarnicki and Peter Pflug, \emph{Invariant distances and metrics in
  complex analysis}, extended ed., De Gruyter Expositions in Mathematics,
  vol.~9, Walter de Gruyter GmbH \& Co. KG, Berlin, 2013. \MR{3114789}

\bibitem[Ker72]{kerzman}
Norberto Kerzman, \emph{The {B}ergman kernel function. {D}ifferentiability at
  the boundary}, Math. Ann. \textbf{195} (1972), 149--158. \MR{0294694 (45
  \#3762)}

\bibitem[Koh72]{kohn}
J.~J. Kohn, \emph{Boundary behavior of {$\delta $} on weakly pseudo-convex
  manifolds of dimension two}, J. Differential Geometry \textbf{6} (1972),
  523--542, Collection of articles dedicated to S. S. Chern and D. C. Spencer
  on their sixtieth birthdays. \MR{0322365}

\bibitem[Kra88]{krantz-comp}
Steven~G. Krantz, \emph{Compactness of the {$\overline\partial$}-{N}eumann
  operator}, Proc. Amer. Math. Soc. \textbf{103} (1988), no.~4, 1136--1138.
  \MR{954995}

\bibitem[Mat98]{matheos}
Peter~George Matheos, \emph{Failure of compactness for the d-bar {N}eumann
  problem for two complex dimensional {H}artogs domains with no analytic disks
  in the boundary}, ProQuest LLC, Ann Arbor, MI, 1998, Thesis
  (Ph.D.)--University of California, Los Angeles. \MR{2698186}

\bibitem[MS94]{mcneal-stein-bergman}
J.~D. McNeal and E.~M. Stein, \emph{Mapping properties of the {B}ergman
  projection on convex domains of finite type}, Duke Math. J. \textbf{73}
  (1994), no.~1, 177--199. \MR{1257282 (94k:32037)}

\bibitem[NP]{nagel-pramanik-diagonal}
Alexander Nagel and Malabika Pramanik, \emph{Diagonal estimates for the
  {B}ergman kernel on certain domains in $\mathbb{C}^n$}, preprint.

\bibitem[Ohs84]{ohsawa-localization}
Takeo Ohsawa, \emph{Boundary behavior of the {B}ergman kernel function on
  pseudoconvex domains}, Publ. Res. Inst. Math. Sci. \textbf{20} (1984), no.~5,
  897--902. \MR{764336}

\bibitem[OT87]{ohsawa-takegoshi}
Takeo Ohsawa and Kensh\=o Takegoshi, \emph{On the extension of {$L^2$}
  holomorphic functions}, Math. Z. \textbf{195} (1987), no.~2, 197--204.
  \MR{892051}

\bibitem[Sah06]{sahutoglu-thesis}
Sonmez Sahutoglu, \emph{Compactness of the delta-{N}eumann problem and {S}tein
  neighborhood bases}, ProQuest LLC, Ann Arbor, MI, 2006, Thesis (Ph.D.)--Texas
  A\&M University. \MR{2708816}

\bibitem[Str10]{straube-book}
Emil~J. Straube, \emph{Lectures on the {$L^2$}-sobolev theory of the
  {$\dbar$}-neumann problem}, ESI Lectures in Mathematics and Physics, European
  Mathematical Society (EMS), Z\"urich, 2010. \MR{2603659}

\end{thebibliography}

\end{document}